\documentclass[12pt]{article}
\usepackage[amsmath]{e-jc}
\usepackage[T1]{fontenc}
\usepackage{lmodern}
\usepackage{mathtools}
\usepackage{microtype,booktabs,array,enumitem,needspace,xcolor,graphicx,placeins}
\hypersetup{pdftitle={Fractional majority coloring of digraphs: bounds and inapproximability},pdfauthor={Walid Ben-Ameur and Alessandro Maddaloni},pdfsubject={Fractional majority coloring, cycle factors, tournaments, and computational complexity}}
\Copyright{The authors.}
\makeatletter
\def\the@dateline{\copy@right}
\def\ps@plain{\def\@oddfoot{\hfil\thepage\hfil}\let\@evenfoot\@oddfoot}
\makeatother
\DeclareMathOperator{\Bin}{Bin}
\newcommand{\Pp}{\mathbb P}
\newcommand{\E}{\mathbb E}
\newcommand{\chif}{\chi^*_{\mathrm{maj}}}
\newcommand{\floor}[1]{\lfloor #1\rfloor}
\newcommand{\abs}[1]{|#1|}
\title{Fractional majority coloring of digraphs:\\bounds and inapproximability}
\author{Walid Ben-Ameur\and Alessandro Maddaloni}
\authortext{}{SAMOVAR, T\'el\'ecom SudParis, Institut Polytechnique de Paris, France
(\email{walid.benameur@telecom-sudparis.eu},
\email{alessandro.maddaloni@telecom-sudparis.eu}).}
\date{}
\begin{document}
\maketitle

\begin{abstract}
A set of vertices in a digraph is majority-stable if each of its vertices has at most half of its outneighbors in the set. A fractional majority coloring assigns nonnegative weights to such sets, covering each vertex to total weight at least one; its minimum total weight is the fractional majority coloring number. We prove that every finite loopless digraph has fractional majority coloring number at most \(523/140<3.736\), improving the bound \(3.9602\) of Anastos, Lamaison, Steiner and Szab\'o. Our construction uses exclusive sampling from pairs in random cycle matchings, followed by deletion and a correction on the acyclic remainder. The same approach yields bounds below \(3.430\) for digraphs with a directed cycle factor, \(3.287\) for those with an even cycle factor, and \(3.324\) for tournaments. As a consequence, for every fixed \(\varepsilon>0\), we obtain a randomized \((2.491+\varepsilon)\)-approximation algorithm producing an explicit fractional majority coloring in expected polynomial time. On the complexity side, we prove NP-completeness of deciding whether the fractional majority coloring number equals $3/2$, and establish a multiplicative inapproximability threshold of $72/71$ and an additive upper-estimation threshold of $3/142$. All three results hold for acyclic oriented digraphs with outdegrees zero or two in which every directed path has length at most two.

\end{abstract}

\section{Introduction}
A majority coloring of a digraph assigns colors to its vertices so that each vertex has its own color on at most half of its outneighbors. For undirected graphs, a partition maximizing the number of crossing edges gives a majority coloring with two colors. The directed version behaves differently: an odd directed cycle requires three colors, and the usual cut argument does not control outgoing edges separately.

Kreutzer, Oum, Seymour, van der Zypen and Wood~\cite{KOSZW} proved that every digraph has a majority coloring with four colors and conjectured that three always suffice. Their proof combines two greedy two-colorings associated with opposite vertex orders. Anholcer, Bosek and Grytczuk~\cite{ABG} extended the four-color theorem to arbitrary lists of four colors. More generally, Gir\~ao, Kittipassorn and Popielarz~\cite{GKP}, and independently Knox and \v{S}\'amal~\cite{KS}, proved that lists of \(2k\) colors suffice when each vertex may share its color with at most a \(1/k\) proportion of its outneighbors. Gir\~ao et al.\ also established the majority three-coloring conjecture for tournaments of sufficiently large minimum outdegree; their numerical refinement gives the threshold 55. Gacit\'ua Fuentes, Jara C\'adiz, Opazo Salazar, Sanhueza-Matamala and Thraves Caro~\cite[Theorems~1.7--1.8]{GJOST} lowered this threshold to 35 and proved that every tournament has a three-coloring with at most four vertices violating the majority condition.

Several structural and probabilistic results support the three-color conjecture. Anastos, Lamaison, Steiner and Szab\'o~\cite{ALSS} proved it for digraphs of chromatic number at most six or dichromatic number at most three, together with corresponding list versions. Here the dichromatic number is the minimum number of acyclic vertex classes. Their results imply majority three-choosability when the maximum outdegree is at most four. On the complexity side, Bang-Jensen, Bessy, Havet and Yeo~\cite{BBHY} proved that recognizing majority two-colorable digraphs is NP-complete, even for strongly connected digraphs in which every outdegree is three. More recently, Anastos, Cooley, Kang and Kwan~\cite{ACKK} proved that Erd\H{o}s--R\'enyi random digraphs admit majority three-colorings with high probability throughout the density range. Anholcer et al.~\cite{MrsCorrect} developed a unified treatment of weighted, list and online variants. Cai, Xia and Yan~\cite[Theorem~3.1]{CXY} proved that every \(n\)-vertex digraph with minimum outdegree greater than \(18\log n\) has a majority three-coloring, and obtained analogous degree conditions for \(1/k\)-majority \((2k-1)\)-coloring~\cite[Theorem~3.2]{CXY}. Gacit\'ua Fuentes et al.~\cite[Theorems~2.1--2.2]{GJOST} also proved majority three-colorability for concave-round digraphs, with two colors sufficient when the number of vertices is even. These digraphs admit a circular order in which each vertex's outneighbors are precisely the next consecutive vertices. Bang-Jensen, Pirot and Yeo~\cite{BPY} investigate majority partitions of two-edge-colored graphs, a related model with a separate majority condition for each edge color.

The fractional version, introduced in~\cite{KOSZW}, has received less attention. A set \(S\subseteq V(D)\) is \emph{majority-stable} if \(\abs{N_D^+(v)\cap S}\le\floor{d_D^+(v)/2}\) for every \(v\in S\). Subsets of majority-stable sets remain majority-stable. Writing \(\mathcal S(D)\) for this family, define
\begin{equation}\label{eq:lp}
\chif(D)=\min\left\{\sum_{S\in\mathcal S(D)}w_S:
 w_S\ge0,\quad \sum_{S\ni v}w_S\ge1\ \text{for every }v\in V(D)\right\}.
\end{equation}
Anastos et al.~\cite{ALSS} proved the universal bound \(\chif(D)\le3.9602\). They also showed that minimum outdegree of order \(\varepsilon^{-2}\log(1/\varepsilon)\) guarantees a fractional majority coloring of weight at most \(2+\varepsilon\). The directed triangle has fractional majority coloring number three, so a universal bound of three would be sharp and would follow from the integral conjecture.

Our upper-bound proof builds on the random-set deletion principle of Anastos et al.~\cite[Lemma~18]{ALSS}. Their construction starts with independent choices and modifies the vertices of outdegree one. Here we impose exclusive choices on disjoint pairs drawn from cycles, giving protection to vertices of arbitrary outdegree, and compensate for the remaining vertices through a partition of the  acyclic remainder into two stable sets. The probabilistic formulation, the deletion principle and the elementary binomial monotonicity are existing ingredients; the bounds below come from this different sampling distribution and its survival analysis.

Our main result uses only random matchings and exclusive choices between their endpoints. We use the same sampling parameter \(b=4/9\) for all the upper bounds below.
\begin{theorem}\label{thm:general}
Every finite loopless digraph \(D\) satisfies \(\chif(D)\le523/140<3.736\).
\end{theorem}
A maximal collection of vertex-disjoint directed cycles leaves an acyclic remainder. On each cycle, a random matching makes every vertex the tail of a matched arc with probability at least one third. We select at most one endpoint of every pair, with the same marginal probability at both endpoints. Thus, when a selected vertex is the tail of a matched arc, one of its outneighbors is certainly absent. This local protection improves its chance of surviving the deletion of vertices with too many selected outneighbors. The acyclic remainder admits two majority-stable classes; mixing with these classes equalizes the guaranteed inclusion probabilities.

If the cycles cover all vertices, no correction is needed. The same sampling parameter gives the following bounds.
\begin{theorem}\label{thm:factors}
If \(D\) has a directed cycle factor, then
\(\chif(D)\le19683/5740<3.430\).
If it has a factor consisting of even directed cycles, then \(\chif(D)\le177147/53900<3.287\).
\end{theorem}
A directed cycle factor is a family of vertex-disjoint directed cycles covering all vertices. The first assertion therefore applies whenever the remainder in our construction is empty, and in particular to every \(r\)-regular digraph with \(r\ge1\). Here \(r\)-regular means that every vertex has both indegree and outdegree equal to \(r\). On concave-round digraphs satisfying the factor hypothesis, the integral results of~\cite{GJOST} give stronger fractional bounds. The factor assumption alone imposes no interval condition on the outneighborhoods.

A further application gives \(\chif(T)<3.324\) for every tournament (Theorem~\ref{thm:tournament}). It uses an outdegree-dependent choice of the endpoint of a Hamilton path; the construction and proof are presented in Appendix~\ref{app:tournament}.

Our complexity result establishes a fixed gap above \(3/2\), the smallest possible value for a digraph containing an arc.
\begin{theorem}\label{thm:complexity}
For every fixed \(1\le\rho<72/71\), approximating the fractional majority coloring number within multiplicative ratio \(\rho\) is NP-hard. For every fixed \(0<\delta<3/142\), computing an upper estimate with additive error at most \(\delta\) is NP-hard. Both statements hold for acyclic oriented digraphs of maximum outdegree two. 
\end{theorem}

The additive inapproximability threshold \(3/142\) concerns upper estimates of the optimum. It follows from the same hardness gap as the multiplicative threshold \(72/71\). The same reduction shows that deciding whether \(\chif(D)=3/2\) is NP-complete, even on this restricted class (Corollary~\ref{cor:threehalves}).

Combining the universal upper bound with the lower bound \(3/2\) gives a deterministic polynomial-time estimate of the optimum within factor \(2.491\). Section~\ref{sec:approximation} gives both this value estimate and a randomized construction of an explicit fractional coloring with expected polynomial running time; its analysis is in Appendix~\ref{app:algorithm}. The hardness proof uses the results of  Austrin, O'Donnell, Tan and Wright~\cite{AOTW}  related to constraint-satisfaction gaps; the reduction to fractional majority coloring and its quantitative soundness analysis are the contributions here.

All digraphs are finite, without loops or parallel arcs; opposite arcs are allowed. A tournament has exactly one directed arc between each pair of distinct vertices. Unless an induced digraph is explicitly specified, degrees and majority thresholds refer to the original digraph. We abbreviate \(d_D^+(v)\) to \(d(v)\) and put \(k(v)=\floor{d(v)/2}\).

\section{The random construction}\label{sec:construction}
\subsection{From a random set to a fractional coloring}
The following distribution interpretation of~\eqref{eq:lp} is standard. Its use with deletion of vertices violating the majority condition is the principle of Anastos et al.~\cite[Lemma~18]{ALSS}.
\begin{lemma}\label{lem:distribution}
If a random majority-stable set \(S\) satisfies \(\Pp(v\in S)\ge p>0\) for every vertex, then \(\chif(D)\le1/p\). Moreover, for every \(K\ge1\), the inequality \(\chif(D)\le K\) holds if and only if there is a random majority-stable set \(S\) satisfying \(\Pp(v\in S)=1/K\) for every vertex.
\end{lemma}
\begin{proof}
For the first assertion, give each majority-stable set \(I\) weight \(\Pp(S=I)/p\). The total weight is \(1/p\), and the weight covering a vertex \(v\) is \(\Pp(v\in S)/p\ge1\). This also proves the reverse implication in the second assertion by taking \(p=1/K\).

Conversely, take a fractional coloring of total weight at most \(K\), add any missing weight to the empty set, and divide all weights by \(K\). This gives a random majority-stable set \(S'\) with marginal \(p_v=\Pp(v\in S')\ge1/K\) at each vertex. Conditional on \(S'\), retain each \(v\in S'\) independently with probability \(1/(Kp_v)\), which belongs to \((0,1]\). The resulting subset \(S\) is majority-stable: removing vertices cannot increase the number of selected outneighbors of a remaining vertex. Finally, \(\Pp(v\in S)=p_v/(Kp_v)=1/K\), as required. 
\end{proof}

\subsection{Why pair vertices?}
Suppose first that every vertex is selected independently with probability \(b\), and a selected vertex is deleted if more than half of its outneighbors were selected. A vertex of outdegree one survives with probability \(b(1-b)\), which is at most \(1/4\). Applying  Lemma~\ref{lem:distribution} leads to the trivial bound \(4\) for the fractional majority coloring number. 
 The issue is that the unique outneighbor of a vertex of outdegree one may be selected at the same time.

Now consider an arc \(x\to y\), and choose \(\{x\}\), \(\{y\}\), or neither with probabilities \(b,b,1-2b\), where \(0<b\le1/2\). Both endpoints still have selection probability \(b\), but they cannot occur together. Conditional on selecting \(x\), its outneighbor \(y\) is absent. This is the protection that pairing provides: one specified outneighbor is guaranteed to be absent when the tail is selected. Protection guarantees survival for a vertex of outdegree one; at larger outdegree, the other outneighbors must still be controlled.

A collection of disjoint pairs partitions the vertex set into \emph{blocks}: each pair is a block of size two, and each vertex belonging to no pair forms a singleton block. Once the pairs are fixed, we make independent random choices on these blocks. Dependence between the two endpoints of a pair is intentional; it does not introduce dependence between different blocks. Directed cycles provide random pairings under which every cycle vertex has a positive probability of being the tail of a paired arc, as illustrated in Figure~\ref{fig:cycle-matchings}. Vertices outside the cycles will be handled using their acyclic structure.

\begin{figure}[!ht]
\centering
\includegraphics[width=\linewidth]{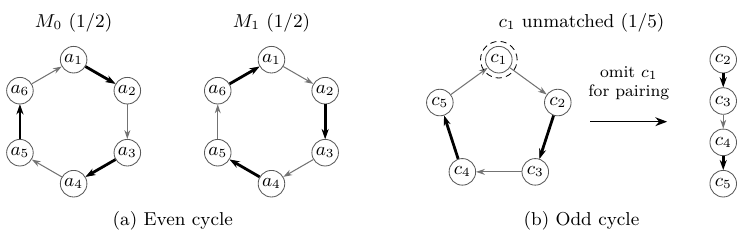}
\caption{Random matchings on directed cycles. Thick arcs are the matched pairs. (a) Choose either alternating perfect matching with probability \(1/2\). (b) Choose an unmatched vertex uniformly (here \(c_1\), with probability \(1/5\)), and match consecutive vertices on the remaining even path. The unmatched vertex remains in the digraph and forms a singleton block for sampling.}
\label{fig:cycle-matchings}
\end{figure}
\FloatBarrier

\subsection{Sampling and deletion}\label{sec:sampling}
Choose an inclusionwise maximal family of vertex-disjoint directed cycles, let \(C\) be their union, and put \(R=V(D)\setminus C\). The digraph \(D[R]\) is acyclic.

Fix \(b=4/9\), the value used in all our sampling constructions. We retain the symbol \(b\) in the probability formulas. Construct a random set \(X\) as follows.
\begin{enumerate}
\item On an even directed cycle, choose one of its two alternating perfect matchings with equal probability. On an odd directed cycle, choose an unmatched vertex uniformly and pair consecutive vertices along the even path obtained by omitting that vertex. Each pair retains the directed cycle arc used to form it. Make these structural choices independently on different cycles; see Figure~\ref{fig:cycle-matchings}.
\item Conditional on these choices, sample independently between blocks. For each pair \(\{x,y\}\), choose \(\{x\}\), \(\{y\}\), or neither with probabilities \(b,b,1-2b\). For each singleton block, choose its vertex with probability \(b\) and choose nothing with probability \(1-b\). The singleton blocks are the unmatched cycle vertices (one vertex per odd cycle) and the vertices of \(R\). Let \(X\) be the union of the chosen subsets of the blocks.
\item Call a selected vertex \(v\) \emph{popular} if it has more than \(k(v) = \lfloor d(v)/2 \rfloor\) outneighbors in \(X\). Delete all popular vertices simultaneously, obtaining
\begin{equation}\label{eq:deletion}
S_0=\{v\in X:\abs{N_D^+(v)\cap X}\le k(v)\}.
\end{equation}
\end{enumerate}
The set \(S_0\) is majority-stable.

\paragraph{Selection probabilities.}
Let \(\mathcal M\) be the random configuration consisting of all chosen matchings and their directed arc labels. Fix a possible configuration \(M\). Every vertex belongs to exactly one block. If \(v\) is in a pair \(\{v,w\}\), then \(v\in X\) exactly when the outcome for this block is \(\{v\}\), which has probability \(b\). If \(v\) is a singleton, its inclusion probability is also \(b\) by construction. Thus \(\Pp(v\in X\mid\mathcal M=M)=b\) for every possible \(M\), regardless of whether \(v\) is paired or unmatched in that configuration. Summing over the configurations gives \(\Pp(v\in X)=\sum_M b\Pp(\mathcal M=M)=b\).

\paragraph{Protection probabilities.}
For a fixed vertex \(v\), let \(P_v\) be the event that the matching pairs \(v\) with one of its outneighbors in \(D\). We say that the matching \emph{protects} \(v\) on this event, and define its \emph{protection probability} by \(q_v=\Pp(P_v)\). If \(P_v\) occurs and \(v\in X\), its mate is an outneighbor that cannot belong to \(X\). Protection depends on the chosen pair and the arcs of \(D\), not on which endpoint is the tail of the designated cycle arc: a pair formed from \(u\to v\) also protects \(v\) whenever \(v\to u\) is an arc of \(D\).

On an even cycle of length at least four, the outgoing cycle arc of \(v\) belongs to exactly one of the two alternating matchings. Hence \(q_v\ge1/2\). On an odd cycle \(v_0\to\cdots\to v_{2m}\to v_0\), the vertex \(v_i\) is a designated matched tail precisely when the unmatched vertex is one of \(v_{i-1},v_{i-3},\ldots,v_{i-(2m-1)}\), with indices modulo \(2m+1\). These are \(m\) distinct choices out of \(2m+1\), so \(q_{v_i}\ge m/(2m+1)\ge1/3\). These lower bounds allow additional protection through reverse arcs.

On a directed two-cycle, or \emph{digon}, both alternating choices give the same underlying pair, with opposite designated arc labels. Whichever choice is made, each vertex is paired with its outneighbor. Thus \(q_v=1\) for each vertex of a selected digon. Vertices of \(R\) are singleton blocks in every configuration, so \(q_v=0\) there.

\subsection{Correcting on the acyclic remainder}
Use the standard greedy majority two-coloring of an acyclic digraph~\cite{KOSZW}: color \(D[R]\) in reverse topological order. When a vertex is reached, all of its outneighbors in \(R\) have been colored; choose a color occurring on at most half of them. This partitions \(R\) into sets \(R_0,R_1\), each majority-stable in \(D[R]\). Each is also majority-stable in \(D\), because it contains no vertex of \(C\) and \(d_{D[R]}^+(v)\le d_D^+(v)\).

Suppose that the deletion analysis gives a survival probability at least \(p_R\) for every vertex of the remainder \(R\), and at least \(p_C\ge p_R\) for every vertex of the cycles \(C\). With probability \(\lambda\), return \(S_0\); otherwise return one of \(R_0,R_1\), uniformly. Call the returned set \(S\). It is always majority-stable. A vertex in \(C\) has inclusion probability at least \(\lambda p_C\), while a vertex in \(R\) has probability at least \(\lambda p_R+(1-\lambda)/2\). Equalizing these two expressions gives
\begin{equation}\label{eq:mixture}
\lambda=\frac{1}{1+2(p_C-p_R)},\qquad
p=\frac{p_C}{1+2(p_C-p_R)},\qquad
\chif(D)\le\frac{1+2(p_C-p_R)}{p_C}.
\end{equation}
Thus \(p_R,p_C\) describe survival before the correction, and \(p\) is the guaranteed inclusion probability in the final set. If \(R\) is empty, we simply return \(S_0\) and use the stronger bound \(1/p_C\).

\begin{table}[!ht]
\centering
\small
\renewcommand{\arraystretch}{1.16}
\begin{tabular}{@{}p{0.19\linewidth}p{0.77\linewidth}@{}}
\toprule
Notation & Meaning in the construction and its analysis \\
\midrule
\(d(v),\ k(v)\) & Outdegree in the original digraph; majority threshold \(k(v)=\floor{d(v)/2}\). \\
\(C,\ R\) & Vertices on the chosen disjoint cycles; the acyclic remainder. \\
\(X,\ S_0\) & Initially selected vertices; vertices surviving simultaneous deletion. \\
\(b=4/9\) & Selection probability of every vertex, for every fixed matching. \\
\(q_v,\ q\) & Actual probability that \(v\) is paired with an outneighbor; a lower bound on that probability. \\
\(L_d(q)\) & Survival lower bound \(L_{d,4/9}(q)\) for outdegree \(d\) and protection at least \(q\); see~\eqref{eq:Ldq}. \\
\(p_R,\ p_C\) & Common survival guarantees on the remainder and on the cycles. \\
\(\lambda,\ p\) & Probability of using \(S_0\) in the mixture; common inclusion guarantee after the mixture. \\
\bottomrule
\end{tabular}
\caption{Recurring notation for the upper-bound proof. Protection is a property of the matching; selection means membership in \(X\); survival means membership in \(S_0\).}
\label{tab:notation}
\end{table}
\FloatBarrier

\section{Survival probabilities}
\label{sec:analysis}

We first justify the probability comparison needed for dependent pairs, then state a general survival bound and minimize it over all outdegrees. The notation \(\Bin(n,p)\) denotes the number of successes in \(n\) independent trials, each with success probability \(p\). In particular, \(\Bin(1,p)\) is the Bernoulli law: it equals one with probability \(p\) and zero otherwise. The probability comparisons and the survival lemma hold for all \(0<b\le1/2\); their numerical applications use \(b=4/9\).

\subsection{A comparison for exclusive pairs}\label{sec:paircomparison}
Fix all matching choices before examining the sampling. For any fixed set \(A\) of vertices, classify the blocks according to their intersection with \(A\). Let \(r\) be the number of pair blocks entirely contained in \(A\), and let \(s\) be the number of blocks meeting \(A\) in exactly one vertex. Since the union of all blocks  equals $V(D)$,  it follows that \(\abs A=2r+s\). A block counted by \(s\) may be a singleton or a pair whose other endpoint lies outside \(A\).

A pair contained in \(A\) contributes one vertex to \(X\cap A\) with probability \(2b\), and contributes zero with probability \(1-2b\): its two endpoints are never chosen together. Its contribution therefore has law \(\Bin(1,2b)\). A block meeting \(A\) in one vertex contributes \(\Bin(1,b)\). Blocks disjoint from \(A\) contribute zero. These contributions are independent because the choices on distinct blocks are independent after the matching configuration is fixed. Consequently \(\abs{X\cap A}\) is the sum of two independent variables with laws \(\Bin(r,2b)\) and \(\Bin(s,b)\).

We will apply this description to the outneighbors of a vertex \(v\), after conditioning also on \(v\in X\). This conditioning only determines the outcome of the block containing \(v\). More explicitly, conditional on a fixed configuration \(M\), let \(Y_B\) be the chosen subset of block \(B\), and let \(B_v\) contain \(v\). The event \(v\in X\) depends only on \(Y_{B_v}\). For any possible outcomes \(y_B\) on the other blocks, independence gives
\[
\Pp\bigl(Y_B=y_B\text{ for all }B\ne B_v\mid\mathcal M=M,\ v\in X\bigr)
=\prod_{B\ne B_v}\Pp(Y_B=y_B\mid\mathcal M=M).
\]
Thus the other blocks keep both their original laws and their mutual independence. If \(v\) has a mate, that mate is absent under this conditioning. Remove it from the outneighbor set if necessary; the remaining outneighbors lie in blocks other than \(B_v\), so their selected count has exactly the two-binomial form above. No independence between individual endpoints in the same pair is asserted. Nor do we need independence after averaging over the random matchings: the comparison below is applied separately to each fixed configuration.

The next lemma bounds the upper tail of this count by the tail obtained from independent selection of all \(2r+s\) vertices. Only thresholds strictly above half of that number are needed.
\begin{lemma}\label{lem:pairtail}
Let \(0<b\le1/2\), let \(r,s\) be nonnegative integers, and let \(Z\) be the sum of independent variables with laws \(\Bin(r,2b)\) and \(\Bin(s,b)\). If \(n=2r+s\) and \(t>n/2\) is an integer, then
\begin{equation}\label{eq:pairtail}
\Pp(Z\ge t)\le\Pp(\Bin(n,b)\ge t).
\end{equation}
\end{lemma}
\begin{proof}
We use induction on \(r\), proving the assertion for every nonnegative \(s\). When \(r=0\), the two sides are equal.

Suppose \(r\ge1\). Write \(Z=W+Q\), where \(Q\sim\Bin(1,2b)\), \(W=I+J\), \(I\sim\Bin(r-1,2b)\), and \(J\sim\Bin(s,b)\), with these variables independent. Let \(Q'\sim\Bin(2,b)\) also be independent of \(W\). We first show that replacing \(Q\) by \(Q'\) cannot decrease the tail at \(t\).

For a fixed value \(w\) of \(W\), independence gives conditional tail probabilities \(\Pp(Q\ge t-w)\) and \(\Pp(Q'\ge t-w)\). If \(w\ge t\), both sums already reach \(t\), so both probabilities are one. If \(w\le t-3\), then \(w+Q\le w+2\le t-1\) and \(w+Q'\le w+2\le t-1\), so neither sum reaches \(t\) and both probabilities are zero. Only \(w=t-1\) and \(w=t-2\) can contribute to the difference. At \(w=t-1\), the difference is \(2b-[1-(1-b)^2]=b^2\); at \(w=t-2\), it is \(0-b^2=-b^2\). Averaging these conditional differences over \(W\), each with weight \(\Pp(W=w)\), therefore gives
\[
\Pp(W+Q\ge t)-\Pp(W+Q'\ge t)
=b^2\bigl(\Pp(W=t-1)-\Pp(W=t-2)\bigr).
\]
To determine the sign, condition on \(I=i\), where \(0\le i\le r-1\), and put \(j=t-1-i\). Since \(t>r+s/2\), we have \(j\ge t-r>s/2\). For \(1\le j\le s\), the binomial mass formula gives
\[
\frac{\Pp(J=j)}{\Pp(J=j-1)}
=\frac{s-j+1}{j}\frac{b}{1-b}\le1.
\]
Indeed, the integer inequality \(j>s/2\) implies \(s-j+1\le j\), and \(b\le1/2\) implies \(b/(1-b)\le1\). If \(j=s+1\), the numerator is zero; if \(j>s+1\), both masses are zero. This also covers \(s=0\), since then \(j\ge1\). In every case \(\Pp(J=j)\le\Pp(J=j-1)\). Multiplying these inequalities by \(\Pp(I=i)\) and summing over \(i\) yields \(\Pp(W=t-1)\le\Pp(W=t-2)\). Substituting this inequality into the tail difference gives
\[
\begin{aligned}
\Pp(W+Q\ge t)-\Pp(W+Q'\ge t)
&=b^2\Pp(W=t-1)-b^2\Pp(W=t-2)\\
&\le b^2\Pp(W=t-2)-b^2\Pp(W=t-2)=0.
\end{aligned}
\]

Finally, \(W+Q'\) is the sum of independent variables with laws \(\Bin(r-1,2b)\) and \(\Bin(s+2,b)\). Since \(2(r-1)+(s+2)=n\), the same threshold satisfies the induction hypothesis. Hence \(\Pp(Z\ge t)\le\Pp(W+Q'\ge t)\le\Pp(\Bin(n,b)\ge t)\), as required.
\end{proof}

We write \(\Phi_d(b)=\Pp(\Bin(d,b)>d/2)\) for the strict-majority tail. We also record a classical binomial majority estimate. The following monotonicity argument is used in Anastos et al.~\cite[Proposition~19]{ALSS}; we include its details for the subsequent minimization.
\begin{lemma}\label{lem:binomial}
For \(0\le b\le1/2\) and every positive integer \(d\),
\(\Phi_d(b)\le b\).
Moreover, \(\Phi_{2j+1}(b)\) is nonincreasing in \(j\), and, for every integer \(j\ge0\),
\begin{equation}\label{eq:binomialdifference}
\Phi_{2j+1}(b)-\Phi_{2j+3}(b)
=\binom{2j+1}{j}[b(1-b)]^{j+1}(1-2b).
\end{equation}
\end{lemma}
\begin{proof}
Let \(Y\sim\Bin(2j+1,b)\), and let \(B_1,B_2\sim\Bin(1,b)\), all independently. Then \(Y+B_1+B_2\sim\Bin(2j+3,b)\). The events defining the two tails are \(A=\{Y\ge j+1\}\) and \(A'=\{Y+B_1+B_2\ge j+2\}\).

The event \(A\setminus A'\) occurs precisely when \(Y=j+1\) and \(B_1=B_2=0\). Indeed, any larger value of \(Y\), or either new success, reaches the new threshold. Conversely, \(A'\setminus A\) occurs precisely when \(Y=j\) and \(B_1=B_2=1\): a smaller value of \(Y\) cannot reach \(j+2\) with only two additional trials. Thus
\[
\begin{aligned}
\Phi_{2j+1}(b)-\Phi_{2j+3}(b)
&=\Pp(Y=j+1)(1-b)^2-\Pp(Y=j)b^2\\
&=\binom{2j+1}{j+1}b^{j+1}(1-b)^{j+2}
-\binom{2j+1}{j}b^{j+2}(1-b)^{j+1}.
\end{aligned}
\]
The binomial coefficients are equal, and factoring proves~\eqref{eq:binomialdifference}. Its right-hand side is nonnegative for \(b\le1/2\), so the odd tails are nonincreasing and bounded by \(\Phi_1(b)=b\).

For even degree \(2j\), with \(j\ge1\), take independent \(Y\sim\Bin(2j,b)\) and \(B\sim\Bin(1,b)\). The event \(Y\ge j+1\) is contained in \(Y+B\ge j+1\). The latter is exactly the strict-majority event for \(2j+1\) trials. Therefore \(\Phi_{2j}(b)\le\Phi_{2j+1}(b)\le b\), completing the proof.
\end{proof}

\needspace{8\baselineskip}
\subsection{A general survival bound}\label{sec:selectedvertex}
\begin{lemma}[Survival under random matchings]\label{lem:survival}
Use the block sampling and deletion of Section~\ref{sec:sampling}
with any $0<b\le1/2$ and any random collection $\mathcal M$ of
disjoint vertex pairs, retaining the notation $X,S_0,q_v$.
Then $S_0$ is majority-stable, every sink survives with probability
$b$, and every vertex $v$ of outdegree $d\ge1$ satisfies
$\Pp(v\in S_0)\ge L_{d,b}(q_v)$, where
\begin{equation}\label{eq:Ldq}
\begin{aligned}
L_{d,b}(q)={}&b(1-q)\,\Pp\bigl(\Bin(d,b)\le\floor{d/2}\bigr)+bq\,\Pp\bigl(\Bin(d-1,b)\le\floor{d/2}\bigr).
\end{aligned}
\end{equation}
For each $d\ge1$, the function $L_{d,b}$ is nondecreasing on $[0,1]$.
\end{lemma}
\begin{proof}
Fix a matching configuration \(M\). The vertex \(v\) is selected with probability \(b\). Conditional on its selection, its mate, if any, is absent, while all other blocks retain their independent sampling distributions. Put \(n=d-1\) if \(M\) protects \(v\), and \(n=d\) otherwise. By Section~\ref{sec:paircomparison}, the selected count among the remaining outneighbors is the sum of independent variables with laws \(\Bin(r,2b)\) and \(\Bin(s,b)\), where \(2r+s=n\). Since the deletion threshold \(\floor{d/2}+1\) is strictly greater than \(n/2\), Lemma~\ref{lem:pairtail} gives
\[
\Pp(v\in S_0\mid\mathcal M=M)
\ge b\,\Pp\bigl(\Bin(n,b)\le\floor{d/2}\bigr).
\]
Averaging over \(M\) proves the stated bound. Removing one independent Bernoulli trial cannot decrease the probability of being at most a fixed threshold, so \(L_{d,b}(q)\) is nondecreasing in \(q\). Majority stability and the assertion for sinks follow directly from the definition of \(S_0\).
\end{proof}

For the remainder of the analysis, set \(b=4/9\) and write \(L_d(q)=L_{d,4/9}(q)\). Lemma~\ref{lem:survival} gives
\[
\Pp(v\in S_0)\ge L_d(q_v)\ge L_d(q)
\qquad\text{whenever }d=d(v)\ge1\text{ and }0\le q\le q_v.
\]
Here \(q_v\) is the actual protection probability; \(q\) is any convenient lower bound on it. Section~\ref{sec:sampling} gives \(q_v\ge1/3\) on every chosen cycle and \(q_v\ge1/2\) on every even chosen cycle. On a selected digon, \(q_v=1\), although the same common lower bounds remain valid. We will also apply Lemma~\ref{lem:survival} to path matchings in Appendix~\ref{app:tournament}.

For a vertex \(v\in R\) of positive outdegree, its block is always a singleton, so \(q_v=0\). Lemmas~\ref{lem:survival} and~\ref{lem:binomial} yield
\[
\Pp(v\in S_0)\ge L_d(0)
=b\bigl(1-\Phi_d(b)\bigr)\ge b(1-b)=\frac{20}{81}.
\]
A sink survives with probability \(b\ge20/81\). Thus we may take \(p_R=20/81\).

\subsection{Minimizing over all outdegrees}\label{sec:degreemin}
We now minimize \(L_d(q)\) over the original outdegree \(d\). Keep \(b=4/9\), and write
\[
R_j=\binom{2j}{j}\left(\frac{20}{81}\right)^j
=\Pp\bigl(\Bin(2j,b)=j\bigr).
\]
This abbreviation is used only in the degree calculations. For odd degree \(d=2j+1\), the permitted count is \(j\). Without protection the bound is \(b(1-\Phi_{2j+1}(b))\). To compute the gain from protection, take independent \(Y\sim\Bin(2j,b)\) and \(B\sim\Bin(1,b)\). Adding the trial \(B\) changes the event that the count is at most \(j\) precisely when \(Y=j\) and \(B=1\). Therefore
\[
\Pp(Y\le j)-\Pp(Y+B\le j)=b\Pp(Y=j)=bR_j.
\]
Multiplying by the selection probability \(b\) and the protection weight \(q\) gives
\begin{equation}\label{eq:Lodd}
L_{2j+1}(q)=b\bigl(1-\Phi_{2j+1}(b)\bigr)+qb^2R_j.
\end{equation}
We compare successive odd degrees using~\eqref{eq:binomialdifference} and the ratios of the binomial coefficients:
\[
\begin{aligned}
\Phi_{2j+1}(b)-\Phi_{2j+3}(b)
&=\frac{2j+1}{j+1}b(1-b)(1-2b)R_j,\\
R_{j+1}&=\frac{2(2j+1)}{j+1}b(1-b)R_j.
\end{aligned}
\]
Indeed, the coefficient ratios are \(\binom{2j+1}{j}/\binom{2j}{j}=(2j+1)/(j+1)\) and \(\binom{2j+2}{j+1}/\binom{2j}{j}=2(2j+1)/(j+1)\). Subtracting the two instances of~\eqref{eq:Lodd} gives
\(b(\Phi_{2j+1}-\Phi_{2j+3})+qb^2(R_{j+1}-R_j)\).
Since \((1-b)(1-2b)=5/81\) and \(2b(1-b)=40/81\), substitution yields
\begin{equation}\label{eq:recurrence}
\begin{aligned}
L_{2j+3}(q)-L_{2j+1}(q)
&=\frac{16R_j}{81}\left[\frac{5(2j+1)}{81(j+1)}
+q\left(\frac{40(2j+1)}{81(j+1)}-1\right)\right]\\
&=\frac{16R_j}{6561(j+1)}\bigl((10-q)j+5-41q\bigr).
\end{aligned}
\end{equation}
Every factor outside the final bracket is positive. The comparison between odd degrees reduces to the sign of a linear expression in \(j\).

Even degrees cannot give a smaller value than the preceding odd degree. For \(i=0,1\), take independent \(Y\sim\Bin(2m-1-i,b)\) and \(B\sim\Bin(1,b)\). If \(Y\le m-1\), then \(Y+B\le m\). Hence
\[
\Pp\bigl(\Bin(2m-i,b)\le m\bigr)
\ge\Pp\bigl(\Bin(2m-1-i,b)\le m-1\bigr).
\]
For \(i=0\) and \(i=1\), these are exactly the two probabilities used in~\eqref{eq:Ldq}. Multiplying by \(b(1-q)\) and \(bq\), respectively, and adding shows that
\(L_{2m}(q)\ge L_{2m-1}(q)\).

We first use \(q=1/3\), the common protection guarantee on all chosen cycles.
\begin{lemma}\label{lem:degreemin}
With \(b=4/9\) as above,
\begin{equation}\label{eq:cycleminimum}
\min_{d\ge1}L_d(1/3)=L_3(1/3)=p_C:=\frac{5740}{19683}.
\end{equation}
\end{lemma}
\begin{proof}
For \(q=1/3\), the final bracket in~\eqref{eq:recurrence} is \((29j-26)/3\). It is negative at \(j=0\) and positive for every \(j\ge1\). Thus the odd-degree values decrease from degree one to degree three and increase thereafter. The even-degree comparison shows that the minimum over all positive degrees is attained at degree three.

To evaluate it, \(\Phi_3(b)=3b^2(1-b)+b^3=3b^2-2b^3\) and \(R_1=2b(1-b)\). Substitution into~\eqref{eq:Lodd} gives
\[
L_3(1/3)=b(1-\Phi_3(b))+\frac{b^2R_1}{3}
=b-\frac73b^3+\frac43b^4
=\frac{5740}{19683},
\]
where the last equality uses \(b=4/9\).
\end{proof}

\begin{proof}[Proof of Theorem~\ref{thm:general}]
Use the construction of Section~\ref{sec:sampling} with \(b=4/9\). Lemma~\ref{lem:survival} and the degree estimates above give the survival guarantees
\[
p_R=\frac{20}{81},\qquad p_C=\frac{5740}{19683},\qquad
p_C-p_R=\frac{880}{19683}>0.
\]
The mixture~\eqref{eq:mixture} therefore applies. It uses \(S_0\) with probability \(\lambda=19683/21443\), and gives every vertex inclusion probability at least
\[
p=\lambda p_C=\frac{5740}{21443}=\frac{140}{523}.
\]
By Lemma~\ref{lem:distribution},
\begin{equation}\label{eq:universalconstant}
\chif(D)\le U:=\frac{1+2(p_C-p_R)}{p_C}=\frac{523}{140}<3.736.
\end{equation}
If \(C\) is empty or \(R\) is empty, the same mixture remains valid, though in those cases a more efficient distribution is available.
\end{proof}

\section{Cycle factors and even cycle factors}\label{sec:factors}
When the selected cycles form a factor, every vertex is protected with probability at least \(1/3\), and the correction on \(R\) disappears. The distinction between even and odd cycles enters only through the protection probability. Having an even cycle factor does not exclude other, odd directed cycles. Thus the even-factor hypothesis differs from the absence of odd directed cycles, for which Anastos et al.~\cite{ALSS} proved the stronger conclusion of majority two-choosability.

\begin{proof}[Proof of the cycle-factor assertion of Theorem~\ref{thm:factors}]
Use the sampling construction with \(b=4/9\) and the given cycle factor. There is no remainder, so Lemmas~\ref{lem:survival} and~\ref{lem:degreemin} give \(\Pp(v\in S_0)\ge p_C=5740/19683\) for every vertex. Returning \(S_0\) without the mixture and applying Lemma~\ref{lem:distribution} yields \(\chif(D)\le1/p_C=19683/5740<3.430\).
\end{proof}

For an even cycle factor, every vertex has protection probability at least \(1/2\). We retain \(b=4/9\) and take \(q=1/2\) in the same survival bound \(L_d(q)\). The recurrence~\eqref{eq:recurrence} becomes
\begin{equation}\label{eq:evenrecurrence}
L_{2j+3}(1/2)-L_{2j+1}(1/2)
=\frac{8R_j}{6561(j+1)}(19j-31).
\end{equation}
The factor \(19j-31\) is negative at \(j=0,1\) and positive for every \(j\ge2\). Thus the odd-degree values decrease through degrees one, three and five, then increase. The even-degree comparison in Section~\ref{sec:degreemin} shows that the minimum over all positive degrees is attained at degree five.

To evaluate it, the binomial formula gives \(\Phi_5(b)=10b^3(1-b)^2+5b^4(1-b)+b^5=10b^3-15b^4+6b^5\), and \(R_2=6b^2(1-b)^2\). Equation~\eqref{eq:Lodd} therefore gives
\begin{equation}\label{eq:evenminimum}
\begin{aligned}
\min_{d\ge1}L_d(1/2)=L_5(1/2)
&=b(1-\Phi_5(b))+\frac{b^2R_2}{2}\\
&=b-7b^4+9b^5-3b^6=\frac{53900}{177147},
\end{aligned}
\end{equation}
where the final equality uses \(b=4/9\).

\begin{proof}[Proof of the even-cycle assertion of Theorem~\ref{thm:factors}]
Use the given even cycle factor and \(b=4/9\). Every vertex is protected with probability at least \(1/2\), so Lemma~\ref{lem:survival} and~\eqref{eq:evenminimum} give \(\Pp(v\in S_0)\ge53900/177147\). There is no remainder to correct. Lemma~\ref{lem:distribution} yields \(\chif(D)\le177147/53900<3.287\).
\end{proof}

\begin{corollary}\label{cor:regular}
Every \(r\)-regular digraph \(D\) with \(r\ge1\) satisfies \(\chif(D)<3.430\).
\end{corollary}
\begin{proof}
The standard cycle-factor consequence of Hall's theorem states that every \(r\)-regular digraph with \(r\ge1\) has a directed cycle factor. Applying Theorem~\ref{thm:factors} gives the bound.
\end{proof}

\paragraph{An application to tournaments.}
\begin{theorem}\label{thm:tournament}
Every tournament \(T\) satisfies \(\chif(T)\le45927/13820<3.324\).
\end{theorem}
The proof uses two alternating matchings of a directed Hamilton path. Every vertex except the terminal vertex is protected with probability \(1/2\); choosing the terminal vertex randomly according to the outdegrees controls the loss of protection at each vertex. This gives a bound for every tournament, without a minimum-outdegree assumption. The complete construction and proof are given in Appendix~\ref{app:tournament}.

\section{Approximation algorithms}\label{sec:approximation}
We use the sampling distribution and a sharp lower bound to construct an explicit fractional coloring with a guaranteed approximation ratio. A multiplicative approximation of ratio \(\rho\ge1\) returns a value \(\widehat K\) satisfying \(\chif(D)\le\widehat K\le\rho\chif(D)\). This convention applies both to upper estimates of the optimum and to the costs of feasible fractional colorings.

\subsection{A sharp lower bound}
The dual of~\eqref{eq:lp} maximizes \(\sum_v y_v\) over nonnegative vertex weights satisfying \(\sum_{v\in S}y_v\le1\) for every majority-stable set \(S\). A weighting concentrated on one vertex and its outneighbors gives the following useful bound.
\begin{proposition}\label{prop:lower}
If \(v\) has positive outdegree \(d\), then \(\chif(D)\ge2-\floor{d/2}/d\ge3/2\). In particular, every digraph containing an arc has value at least \(3/2\), and the inequality is strict if some vertex has positive odd outdegree.
\end{proposition}
\begin{proof}
Put \(k=\floor{d/2}\). Give \(v\) weight \((d-k)/d\), each of its \(d\) outneighbors weight \(1/d\), and every other vertex weight zero. These vertices are distinct because there are no loops. A majority-stable set omitting \(v\) has weight at most one. A set containing \(v\) contains at most \(k\) of its outneighbors, so its weight is at most \((d-k)/d+k/d=1\). The weighting is dual feasible and has total weight \(2-k/d\). Weak duality proves the claim.
\end{proof}
The bound is sharp: in the digraph with arcs \(v\to x\) and \(v\to y\) and no others, all three two-vertex subsets are majority-stable. Giving each weight \(1/2\) proves that the value is \(3/2\). A nonempty edgeless digraph has value one; the empty digraph has value zero.

\subsection{Constructing an explicit fractional coloring}\label{sec:positiveapprox}
Let \(U=523/140\) be the constant in~\eqref{eq:universalconstant}, and put
\begin{equation}\label{eq:approxratio}
\alpha=\frac{2U}{3}=\frac{523}{210}<2.491.
\end{equation}
For a digraph containing an arc, Theorem~\ref{thm:general} and Proposition~\ref{prop:lower} give \(\chif(D)\le U\le\alpha\chif(D)\). Thus returning \(U\), and the exact values zero or one in the empty and nonempty edgeless cases, gives a deterministic polynomial-time value estimate with ratio \(\alpha\). The sampling distribution yields the following constructive guarantee.
\begin{corollary}\label{cor:valueapprox}
For every fixed rational $\varepsilon>0$, there is a randomized algorithm that returns a feasible fractional majority coloring of cost at most $(\alpha+\varepsilon)\chif(D)$ and runs in expected polynomial time.
\end{corollary}
\begin{proof}
The randomized construction below has the stated approximation guarantee and expected polynomial running time, as proved in Appendix~\ref{app:algorithm}.
\end{proof}

In the same spirit as the sampling arguments of Alth\"ofer~\cite{Althofer} and Lipton and Young~\cite{LiptonYoung}, we give the following randomized construction of an explicit fractional majority coloring. The empty digraph needs no sets; for a nonempty edgeless digraph, give \(V(D)\) weight one. Otherwise let \(n=\abs{V(D)}\). Fix a maximal family of disjoint directed cycles and the two classes of its acyclic remainder once and for all. Use the sampler from Theorem~\ref{thm:general}, including the mixture~\eqref{eq:mixture}, with \(b=4/9\). A sampled majority-stable set \(S\) then satisfies \(\Pp(v\in S)\ge p:=1/U=140/523\) for every vertex. For the desired accuracy, choose the target total weight $T$ and the sampling parameters as follows:
\begin{equation}\label{eq:batchparameters}
T=U+\frac{3\varepsilon}{2},\qquad
\Delta=p-\frac1T>0,\qquad
m=1+\left\lfloor\frac{n}{2\Delta^2}\right\rfloor.
\end{equation}
\begin{enumerate}
\item Sample \(m\) independent sets \(S_1,\ldots,S_m\). Let \(C_v=\abs{\{i:v\in S_i\}}\) and \(C_{\min}=\min_v C_v\).
\item If \(C_{\min}<m/T\), discard the batch and repeat the first step using fresh randomness.
\item Otherwise give each sampled set weight \(1/C_{\min}\), adding the weights when a set occurs more than once, and return the resulting fractional coloring.
\end{enumerate}
Every accepted batch covers every vertex and has total weight \(m/C_{\min}\le T\). The success probability of a batch is greater than \(1/2\), so the expected number of batches is less than two. The proof in Appendix~\ref{app:algorithm} includes the independence argument, the variance estimate and the expected polynomial running-time bound.

Any improved universal upper bound would directly strengthen both \(\alpha\) and the constructive guarantee above.

\section{Inapproximability}\label{sec:complexity}
We prove Theorem~\ref{thm:complexity} by a reduction with perfect completeness: satisfiable instances give fractional majority coloring number exactly \(3/2\), whereas instances with a fixed proportion of unavoidable violations give a value bounded away from \(3/2\). The quantitative relation between these two quantities is the essential part of the argument. We use the multiplicative approximation convention of Section~\ref{sec:approximation}. An additive approximation with error \(\delta\ge0\) returns an upper estimate \(\widehat K\) satisfying \(\chif(D)\le\widehat K\le\chif(D)+\delta\).

\subsection{The source hardness theorem}
An instance \(\mathcal I\) of \(\mathrm{2NLin}(\mathbb Z_3)\) consists of a finite set of variables and a list of \(m\ge1\) constraints of the form \(e=(u,v,c_e)\), where \(u\ne v\) are variables and \(c_e\in\mathbb Z_3\). An assignment \(a\) gives each variable a label in \(\mathbb Z_3\); it satisfies \(e\) if \(a(u)-a(v)\ne c_e\). All label arithmetic is modulo three. Each constraint has a positive rational weight \(w_e\), and we normalize these weights so that \(\sum_e w_e=1\). We write
\[
\mu(\mathcal I)=\min_a\sum_{e=(u,v,c_e)}w_e\mathbf1_{\{a(u)-a(v)=c_e\}}
\]
for the minimum total weight of violated constraints. In particular, \(\mu(\mathcal I)=0\) means that the instance is satisfiable.

The hardness theorem of Austrin, O'Donnell, Tan and Wright~\cite[Theorem~1.4]{AOTW} implies that, for every fixed \(0<\eta<1/12\), it is NP-hard to distinguish
\[
\mu(\mathcal I)=0
\quad\text{from}\quad
\mu(\mathcal I)>\frac1{12}-\eta.
\]
We transfer this gap to fractional majority coloring by a construction denoted \(D_2(\mathcal I)\). We will prove that satisfiability gives value exactly \(3/2\), and that a fractional majority coloring of value \(K\) yields an assignment with total violated weight at most \(6-9/K\). Thus unavoidable violations force the fractional majority coloring number above \(3/2\).

\subsection{The construction and its soundness}\label{sec:D2}
For each variable \(v\), introduce three vertices \(x_{v,0},x_{v,1},x_{v,2}\), one for each possible label, and the two arcs \(x_{v,0}\to x_{v,1}\) and \(x_{v,0}\to x_{v,2}\). These are the only outgoing arcs of the variable vertices. A majority-stable set cannot contain all three vertices: \(x_{v,0}\) would then have both outneighbors in the set. At least one vertex of every variable triple must therefore lie in its complement. Here and below, an omitted vertex means a vertex in the complement \(J=V(D_2)\setminus S\) of the majority-stable set \(S\). We interpret the indices of the omitted vertices of a triple as a nonempty list of available labels.

For each constraint \(e=(u,v,c_e)\), introduce three further vertices \(z_{e,0},z_{e,1},z_{e,2}\), with outgoing neighborhoods
\begin{equation}\label{eq:D2}
N^+(z_{e,i})=\{x_{u,i},x_{v,i-c_e}\}\qquad(i\in\mathbb Z_3).
\end{equation}
Together with the variable arcs above, these are all the arcs of \(D_2(\mathcal I)\). Distinct endpoints ensure that the two outneighbors in~\eqref{eq:D2} are distinct.

The construction takes polynomial time: with \(n\) variables and \(m\) constraints it has \(3n+3m\) vertices and \(2n+6m\) arcs. It is a simple oriented acyclic digraph with outdegrees zero or two. Indeed, all arcs go from auxiliary vertices to variable vertices, or from a variable center \(x_{v,0}\) to its two sink leaves. This also shows that every directed path has length at most two.

\begin{proposition}\label{prop:D2}
If \(\mathcal I\) is satisfiable, then \(\chif(D_2(\mathcal I))=3/2\). For every instance,
\begin{equation}\label{eq:D2soundness}
\chif(D_2(\mathcal I))\ge\frac9{6-\mu(\mathcal I)}.
\end{equation}
\end{proposition}
\begin{proof}
Let \(a\) satisfy every constraint. For each \(t\in\mathbb Z_3\), form a set \(J_t\) of vertices omitted from \(S_t=V(D_2)\setminus J_t\) as follows. From each variable triple omit exactly \(x_{v,a(v)+t}\). For a constraint \(e=(u,v,c_e)\), the two labels \(a(u)+t\) and \(a(v)+c_e+t\) are distinct because \(a\) satisfies \(e\). Let \(h_e(t)\) be the unique third label, and omit exactly \(z_{e,h_e(t)}\) from its auxiliary triple. These are all the vertices of \(J_t\). Figure~\ref{fig:constraint} illustrates the construction for one constraint and \(t=0\).

\begin{figure}[!ht]
\centering
\includegraphics[width=0.60\linewidth]{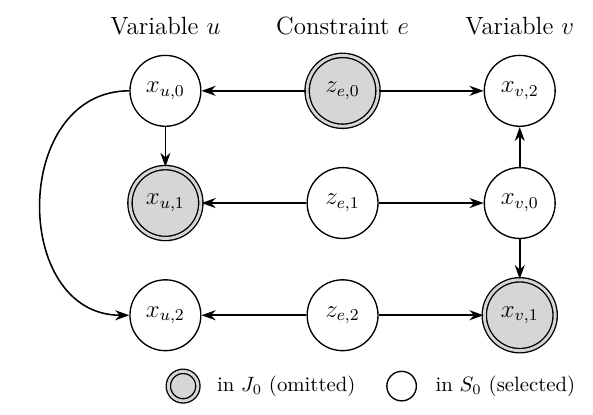}
\caption{The digraph for two variables \(u,v\) and one constraint \(e=(u,v,1)\). The assignment \(a(u)=a(v)=1\) satisfies the constraint. For \(t=0\), the two labels used in the proof are \(a(u)=1\) and \(a(v)+1=2\), so \(h_e(0)=0\). The shaded, double-bordered vertices form \(J_0=\{x_{u,1},x_{v,1},z_{e,0}\}\); the white vertices form \(S_0\). Each white vertex of positive outdegree has exactly one shaded and one white outneighbor.}
\label{fig:constraint}
\end{figure}

We check majority stability at every type of vertex in \(S_t\). If a variable center \(x_{v,0}\) is selected, the omitted vertex of its triple is one of its two leaves. Exactly one of its two outneighbors is therefore selected. For a selected auxiliary vertex \(z_{e,i}\), we have \(i\ne h_e(t)\), so \(i\) equals exactly one of \(a(u)+t\) and \(a(v)+c_e+t\). The first equality says that \(x_{u,i}\) is omitted, and the second says that \(x_{v,i-c_e}\) is omitted. Exactly one of its two outneighbors is consequently selected. All other selected vertices are sinks. Hence \(S_t\) is majority-stable.

As \(t\) runs through \(\mathbb Z_3\), each variable vertex is omitted exactly once. Translation permutes the three labels, so \(h_e(t)=h_e(0)+t\); each auxiliary vertex is also omitted exactly once. Every vertex belongs to exactly two of \(S_0,S_1,S_2\). Giving each of these sets weight \(1/2\) is therefore a fractional coloring of total weight \(3/2\). The digraph contains an arc, so Proposition~\ref{prop:lower} proves equality.

\needspace{5\baselineskip}
For the second assertion, let $K=\chif(D_2(\mathcal I))$. We have \(K\ge3/2\) by Proposition~\ref{prop:lower}. Lemma~\ref{lem:distribution} gives a random majority-stable set \(S\) with every inclusion probability exactly \(1/K\). Put \(J=V(D_2)\setminus S\). Every vertex is omitted with probability \(1-1/K\). Define
\begin{equation}\label{eq:tau}
\tau=3\left(1-\frac1K\right)-1=2-\frac3K.
\end{equation}
For every variable \(v\), let \(L_v=\{i\in\mathbb Z_3:x_{v,i}\in J\}\). This list is nonempty in every outcome: if it were empty, \(S\) would contain \(x_{v,0}\) and both of its outneighbors, contradicting majority stability. Exact marginals give \(\E\abs{L_v}=3(1-1/K)=1+\tau\), hence \(\E(\abs{L_v}-1)=\tau\). Thus \(\tau\) is the expected number of extra available labels beyond the one that every variable must have. In particular, when \(K=3/2\), all lists are singletons almost surely. When \(K\) is close to \(3/2\), the expected number of extra labels is small.

Sample \(J\), and then, conditional on \(J\), choose a label \(a(v)\) uniformly from \(L_v\), independently for different variables. We do not assume that the lists, or the membership indicators of \(J\), are independent. For a constraint \(e=(u,v,c_e)\), write \(L_v+c_e=\{j+c_e:j\in L_v\}\). Because \(u\ne v\), the two final choices \(a(u)\) and \(a(v)+c_e\) are conditionally independent and uniform on \(L_u\) and \(L_v+c_e\). Of the \(\abs{L_u}\abs{L_v}\) equally likely ordered pairs, precisely \(\abs{L_u\cap(L_v+c_e)}\) have equal coordinates. Therefore
\begin{equation}\label{eq:listviolation}
\Pp\bigl(a(u)-a(v)=c_e\mid J\bigr)
=\frac{\abs{L_u\cap(L_v+c_e)}}{\abs{L_u}\abs{L_v}}.
\end{equation}

To bound the numerator, fix a label \(i\in\mathbb Z_3\) and consider the triple
\(\{x_{u,i},x_{v,i-c_e},z_{e,i}\}\).
At least one of these vertices is omitted: otherwise \(z_{e,i}\in S\) would have both outneighbors in \(S\). Thus the number of omitted vertices in this triple, minus one, is nonnegative. Moreover, if \(i\in L_u\cap(L_v+c_e)\), both variable vertices of the triple are omitted, so this number is at least one. Summing these pointwise observations over the three labels gives
\[
\abs{L_u\cap(L_v+c_e)}
\le\sum_{i\in\mathbb Z_3}
\left(\abs{J\cap\{x_{u,i},x_{v,i-c_e},z_{e,i}\}}-1\right).
\]
Each triple has three distinct vertices, each omitted with probability \(1-1/K\). By linearity of expectation, the expected value of each summand on the right is \(3(1-1/K)-1=\tau\). Independence is not needed. Hence
\(\E\abs{L_u\cap(L_v+c_e)}\le3\tau\).
Since both lists are nonempty, their product of sizes is at least one. Taking expectations in~\eqref{eq:listviolation} now gives
\[
\begin{aligned}
\Pp\bigl(a(u)-a(v)=c_e\bigr)
&=\E\left[\frac{\abs{L_u\cap(L_v+c_e)}}{\abs{L_u}\abs{L_v}}\right]\\
&\le\E\abs{L_u\cap(L_v+c_e)}\le3\tau.
\end{aligned}
\]
Multiplying the inequality for constraint \(e\) by its weight \(w_e\), summing over all constraints, and using \(\sum_e w_e=1\), we obtain
\[
\E\left[\sum_{e=(u,v,c_e)}w_e\mathbf1_{\{a(u)-a(v)=c_e\}}\right]
=\sum_{e=(u,v,c_e)}w_e\Pp\bigl(a(u)-a(v)=c_e\bigr)
\le3\tau.
\]
At least one assignment has total violated weight no greater than this expectation. By the definition of \(\mu\), we conclude that \(\mu(\mathcal I)\le3\tau=6-9/K\). Consequently \(9/K\le6-\mu(\mathcal I)\); the right-hand side is positive since \(0\le\mu(\mathcal I)\le1\). Rearranging proves~\eqref{eq:D2soundness}.
\end{proof}

\FloatBarrier
\needspace{6\baselineskip}
\begin{corollary}\label{cor:threehalves}
Deciding whether a digraph $D$ satisfies $\chif(D)=3/2$ is NP-complete, even for acyclic oriented digraphs with outdegrees zero or two in which every directed path has length at most two.
\end{corollary}
\begin{proof}
NP-hardness follows from Proposition~\ref{prop:D2} and the source hardness theorem. For membership in NP, reject edgeless inputs; otherwise Proposition~\ref{prop:lower} gives $\chif(D)\ge3/2$. The covering LP~\eqref{eq:lp} has a basic optimal solution with at most $\abs{V(D)}$ positive weights, all rational and of polynomial bit length. A family of majority-stable sets with these weights and total weight at most $3/2$ is therefore a polynomial certificate: majority stability, vertex coverage and total weight can all be checked in polynomial time.
\end{proof}

\subsection{Multiplicative and additive inapproximability thresholds}\label{sec:gaps}
Both thresholds follow from one hardness gap: dividing its endpoints gives the multiplicative threshold, while subtracting them gives the additive upper-estimation threshold.
\begin{proof}[Proof of Theorem~\ref{thm:complexity}]
Proposition~\ref{prop:D2} maps a satisfiable source instance to value \(a=3/2\), and a NO instance with \(\mu(\mathcal I)>1/12-\eta\) to a value greater than \(108/(71+12\eta)\). Thus, for every fixed \(\beta\) with \(a<\beta<B:=108/71\), choosing a sufficiently small fixed \(\eta>0\) gives NP-hardness of distinguishing value \(a\) from a value greater than \(\beta\), on acyclic oriented digraphs of maximum outdegree two.

A multiplicative approximation of ratio \(\rho<B/a=72/71\) would separate these cases: choose \(\rho a<\beta<B\), so its output is at most \(\rho a\) in the YES case and greater than \(\beta\) in the NO case. Likewise, for an additive upper estimate with error \(\delta<B-a=3/142\), choose \(a+\delta<\beta<B\); its output is then at most \(a+\delta\) in the YES case and greater than \(\beta\) in the NO case. In either case a fixed rational threshold between the two output ranges would solve the NP-hard promise problem in polynomial time.
\end{proof}

\section{Concluding remarks}
The common choice $b=4/9$ allows us to treat general digraphs,
digraphs with a directed cycle factor, and tournaments within the
same sampling framework. We have favored a uniform argument over
optimizing the constants separately. A possible natural refinement would be to define some local sampling  rules
adapted to the length of an odd cycle in the family of vertex-disjoint cycles.

The central extremal question remains whether $\chif(D)\le3$
for every digraph $D$. Such a bound would be sharp: a terminal
strongly connected component isomorphic to a directed triangle
forces $\chif(D)\ge3$, since every majority-stable set contains
at most one of its vertices.

We ask whether every digraph with no such component satisfies
$\chif(D)\le8/3$. This constant would also be sharp: the digraph on
$\{a,b,c_1,c_2,c_3\}$ with arcs $a\to b$, $b\to c_i$ and
$c_i\to a$ for $i=1,2,3$ has fractional majority coloring number
$8/3$. Indeed, the vertex weights $w(a)=1$, $w(b)=2/3$ and
$w(c_i)=1/3$ give total weight at most one to every majority-stable
set, proving the lower bound. A matching fractional coloring
assigns weight $1$ to $\{a\}$, weight $1/3$ to each $\{b,c_i\}$,
and weight $2/3$ to $\{c_1,c_2,c_3\}$.

The constructive and hardness results leave open the optimal
polynomial-time approximation ratio for fractional majority
coloring. The hardness already holds for acyclic oriented digraphs
of maximum outdegree two, making this restricted class a natural
setting in which to study that question.

\clearpage
\appendix
\section{Tournaments: construction and proof}\label{app:tournament}
A directed Hamilton path gives another way to form pairs. For a path \(v_1\to\cdots\to v_n\), let \(M_0\) consist of its arcs \(v_i\to v_{i+1}\) with odd \(i\), and let \(M_1\) consist of those with even \(i\). Choose between these matchings with equal probability. Every vertex other than \(v_n\) is the tail of a matched arc in exactly one choice, whereas \(v_n\) is never a matched tail. Thus a fixed endpoint receives no protection through a designated path arc. We randomize the endpoint and control each vertex's probability of being chosen through its outdegree.

For a nonempty tournament \(T\), let \(K\) be its terminal strong component and put \(Z=\sum_{v\in K}(2d(v)+1)^{-1}\). The procedure is as follows; the proof below justifies the Hamilton paths and the probability estimates.
\begin{enumerate}
\item Choose an endpoint \(z\in K\) with probability \(1/(Z(2d(z)+1))\). For each possible \(z\), fix a directed Hamilton path of \(T\) ending at \(z\), and use that path when \(z\) is chosen.
\item Choose \(M_0\) or \(M_1\) along this path with probability \(1/2\). Each matched pair forms a block of size two, and each unmatched vertex forms its own block of size one.
\item Apply the block sampling and deletion of Section~\ref{sec:sampling} to these blocks, using the outdegrees in $T$, to obtain $S_0$.
\end{enumerate}
Only the first and last vertices of the path can be unmatched. If \(n\) is even, \(M_0\) covers every vertex, whereas \(M_1\) leaves exactly \(v_1\) and \(v_n\) unmatched. If \(n\) is odd and at least three, \(M_0\) leaves only \(v_n\) unmatched, and \(M_1\) leaves only \(v_1\) unmatched. For \(n=1\), both matchings are empty and the sole vertex is a singleton. These unmatched endpoints, and no internal vertices, form the singleton blocks.

Conditional on the chosen path, every vertex has selection probability $b$, as in Section~\ref{sec:sampling}. Every vertex other than $v_n$ has protection probability $1/2$, while $v_n$ is never protected.

We again use \(b=4/9\). These block choices and the simultaneous deletion are exactly those of Lemma~\ref{lem:survival}, now applied to the distribution of path matchings. There is no separate correction on an acyclic remainder.

\begin{proof}[Proof of Theorem~\ref{thm:tournament}]
The empty tournament has value zero, so assume \(T\) is nonempty. Write \(s=\abs K\). The strong components of a tournament are linearly ordered, with every arc between two components pointing in the same direction. Thus every vertex outside \(K\) dominates every vertex in \(K\), and the outdegrees of vertices in \(K\) equal their outdegrees within \(K\).

Every \(z\in K\) can be the last vertex of a directed Hamilton path of \(T\): indeed, the subtournament outside \(K\) has a Hamilton path, possibly empty. If \(K\) is a singleton, append it. Otherwise, there exists a Hamilton cycle in \(K\), break its outgoing arc at \(z\) to obtain a Hamilton path ending at \(z\), and append this path to the one outside \(K\).

There are exactly \(\binom{s}{2}\) arcs inside \(K\), and none leaving \(K\). Consequently \(\sum_{v\in K}(2d(v)+1)=2\binom{s}{2}+s=s^2\). Apply Cauchy--Schwarz to the vectors with coordinates \(\sqrt{2d(v)+1}\) and \(1/\sqrt{2d(v)+1}\). It gives
\[
s^2=\left(\sum_{v\in K}\sqrt{2d(v)+1}\,
\frac1{\sqrt{2d(v)+1}}\right)^2
\le\left(\sum_{v\in K}(2d(v)+1)\right)
\left(\sum_{v\in K}\frac1{2d(v)+1}\right)
=s^2Z.
\]
Since \(s>0\), this proves \(Z\ge1\).

Let \(r_v\) denote the probability that \(v\) is chosen as the endpoint: \(r_v=1/(Z(2d(v)+1))\) for \(v\in K\), and \(r_v=0\) otherwise. Thus \(r_v\le1/(2d(v)+1)\) for every vertex. For a fixed choice of the endpoint, a vertex distinct from it is a designated matched tail with probability \(1/2\), while the endpoint never is. The protection event \(P_v\) has probability \(q_v=(1-r_v)/2\).

For a vertex of outdegree \(d\ge1\), we therefore have
\(q_v\ge(1-1/(2d+1))/2=d/(2d+1)\).
Applying Lemma~\ref{lem:survival} with this protection guarantee gives \(\Pp(v\in S_0)\ge h_d:= L_d\left(\frac{d}{2d+1}\right) \). A sink is never deleted and survives with probability \(b=4/9\).

It remains to minimize \(h_d\). The even-degree comparison in Section~\ref{sec:degreemin} gives \(L_{2m}(q)\ge L_{2m-1}(q)\), and \(L_d(q)\) is nondecreasing in \(q\). Since \(2m/(4m+1)\ge(2m-1)/(4m-1)\),
\[
\begin{aligned}
h_{2m}=L_{2m}\left(\frac{2m}{4m+1}\right)
&\ge L_{2m-1}\left(\frac{2m}{4m+1}\right)\\
&\ge L_{2m-1}\left(\frac{2m-1}{4m-1}\right)=h_{2m-1}.
\end{aligned}
\]
It is therefore enough to compare odd degrees.

Recall \(R_j=\binom{2j}{j}(20/81)^j\) and \(\Phi_d(b)=\Pp(\Bin(d,b)>d/2)\), with \(b=4/9\). Applying~\eqref{eq:Lodd} with \(q=(2j+1)/(4j+3)\) gives
\[
h_{2j+1}=b(1-\Phi_{2j+1}(b))+b^2\frac{2j+1}{4j+3}R_j.
\]
For \(m\ge1\), subtracting the expressions for consecutive odd degrees and using~\eqref{eq:binomialdifference} together with
\(R_m=\frac{40(2m-1)}{81m}R_{m-1}\) gives
\[
h_{2m+1}-h_{2m-1}
=R_{m-1}\frac{16(2m-1)(76m^2-123m-55)}{6561m(4m+3)(4m-1)}.
\]
Every factor except the quadratic is positive. Moreover, $76m^2-123m-55=3+(m-2)(76m+29)$, which is
\(-102\) at \(m=1\), while
it is positive, for $m\ge 2$.
Thus \(h_3<h_1\), and the odd-degree values increase strictly from
degree three onwards. Together with the even-degree comparison,
this proves that \(h_d\ge h_3\) for every \(d\ge1\).

Finally, the definition~\eqref{eq:Ldq} gives, at \(b=4/9\),
\[
h_3=L_3(3/7)
=b\left[\frac47\bigl((1-b)^3+3b(1-b)^2\bigr)
+\frac37(1-b^2)\right]
=\frac{13820}{45927}.
\]
The sink probability \(4/9\) is larger. Every vertex therefore belongs to \(S_0\) with probability at least \(13820/45927\), and Lemma~\ref{lem:distribution} gives \(\chif(T)\le45927/13820<3.324\).
\end{proof}

\section{Constructing an explicit fractional majority coloring}\label{app:algorithm}
We prove Corollary~\ref{cor:valueapprox} by analyzing the randomized algorithm in Section~\ref{sec:positiveapprox}. The empty and edgeless cases are exact, so assume that \(D\) contains an arc. Recall the parameters from~\eqref{eq:batchparameters}: \(n\) is the number of vertices, \(p=140/523\) is the guaranteed inclusion probability in one sample, \(T\) is the target cost, \(\Delta=p-1/T\) is the coverage margin, and \(m\) is the number of samples in a batch.

\paragraph{The sampling distribution.}
A maximal family of disjoint directed cycles can be obtained by repeatedly finding a directed cycle in the remaining induced digraph and deleting its vertices. A cycle search takes polynomial time, and there are at most \(n\) iterations. The remaining digraph is acyclic, and reverse topological greedy coloring gives \(R_0,R_1\). This preprocessing is performed once. Each sample independently chooses the branch of~\eqref{eq:mixture}, then either returns \(R_0\) or \(R_1\), or performs the random matching, block sampling and deletion of Section~\ref{sec:sampling}. With \(b=4/9\), all branch and block probabilities are fixed rationals, and an odd cycle of length \(\ell\) requires a uniform choice among \(\ell\) unmatched vertices. Each sample is majority-stable and has marginal \(p_v:=\Pp(v\in S)\ge p\) at every vertex, by the proof of Theorem~\ref{thm:general}.

\paragraph{Coverage in one batch.}
For a fixed vertex \(v\), the indicators \(\mathbf1_{\{v\in S_i\}}\), \(1\le i\le m\), are independent Bernoulli variables with parameter \(p_v\), because the complete samples use independent randomness. Dependence between different vertices within one sample is irrelevant here. Therefore \(C_v\) has law \(\Bin(m,p_v)\), with \(\E C_v=mp_v\) and
\[
\operatorname{Var}(C_v)=mp_v(1-p_v)
=m\left(\frac14-(p_v-\tfrac12)^2\right)\le\frac m4.
\]
We use Chebyshev's inequality in the form
\[
\Pp\bigl(\abs{Y-\E Y}\ge t\bigr)
\le\frac{\operatorname{Var}(Y)}{t^2}\qquad(t>0),
\]
valid for any random variable \(Y\) with finite variance. If \(C_v<m/T\), then \(\E C_v-C_v>m(p_v-1/T)\ge m\Delta\), since \(p_v\ge p\) and \(\Delta=p-1/T\). Hence the event \(\{C_v<m/T\}\) is contained in \(\{\abs{C_v-\E C_v}\ge m\Delta\}\). Applying the inequality with \(Y=C_v\) and \(t=m\Delta>0\) gives
\[
\Pp(C_v<m/T)
\le\Pp\bigl(\abs{C_v-\E C_v}\ge m\Delta\bigr)
\le\frac{\operatorname{Var}(C_v)}{m^2\Delta^2}
\le\frac1{4m\Delta^2}.
\]
A union bound over the \(n\) vertices now yields
\begin{equation}\label{eq:batchfailure}
\Pp(C_{\min}<m/T)\le\frac{n}{4m\Delta^2}<\frac12,
\end{equation}
where the last inequality follows from \(m>n/(2\Delta^2)\). No independence between the counts for different vertices is needed for this union bound.

\paragraph{The returned fractional coloring.}
On acceptance, \(C_{\min}\ge m/T>0\). Assigning each sampled set weight \(1/C_{\min}\) covers a vertex \(v\) to total weight \(C_v/C_{\min}\ge1\). All sampled sets are majority-stable, so the result is feasible. Repeated sets can be merged without changing any coverage or the total weight. The cost is \(m/C_{\min}\le T\); since \(\chif(D)\ge3/2\), its approximation ratio is at most
\[
\frac{T}{\chif(D)}\le\frac{2T}{3}
=\frac{2U}{3}+\varepsilon=\alpha+\varepsilon.
\]
The output uses at most \(m\) sets. Their weights have denominator \(C_{\min}\le m\), and after merging, their numerators are also at most \(m\). Thus the output has polynomial encoding length whenever \(m\) is polynomial.

\paragraph{Termination and implementation.}
Batches are independent and have the same success probability \(s>1/2\), by~\eqref{eq:batchfailure}. If \(B\) is the number of batches used, then \(\Pp(B>j)=(1-s)^j\), and hence \(\E B=\sum_{j\ge0}(1-s)^j=1/s<2\). The algorithm therefore terminates with probability one and only returns feasible colorings satisfying the claimed bound.

For fixed rational \(\varepsilon>0\), \(\Delta\) is a positive constant and \(m=O(n)\). More generally, \(\Delta=3\varepsilon/[2U(U+3\varepsilon/2)]\), so \(m=O(n(1+\varepsilon^{-2}))\). Matching selection, block sampling, deletion and vertex counting take
expected polynomial time per batch. Since batches use independent
randomness and $\E B<2$, the expected total running time is polynomial.

\end{document}